\documentclass[preprint, 12pt]{elsarticle}

\usepackage{amssymb}
\usepackage{amsmath}
 \usepackage{amsthm}
 \usepackage{tikz -inet}
\usetikzlibrary{positioning,shapes.misc, patterns,arrows,decorations,decorations.pathreplacing, snakes}

\usetikzlibrary{arrows,positioning} 
\tikzset{
    >=stealth',
    punkt/.style={
           rectangle,
           rounded corners,
           draw=black, very thick,
           text width=8em,
           minimum height=2em,
           text centered},
    pil/.style={
           ->,
           thick,
           shorten <=2pt,
           shorten >=2pt,}
}

\newtheorem{theorem}{Theorem}

\newdefinition{definition}[theorem]{Definition}
\newtheorem{proposition}[theorem]{Proposition}

\newtheorem{remark}[theorem]{Remark}
\newtheorem{corollary}[theorem]{Corollary}
\newtheorem{fact}[theorem]{Fact}
\newtheorem{claim}[theorem]{Claim}
\newcommand{\K}{\mathcal{K}}
\newcommand{\Union}{\bigcup}
\DeclareMathOperator{\tp}{ga-tp}
\DeclareMathOperator{\id}{id}
\DeclareMathOperator{\Aut}{Aut}
\DeclareMathOperator{\cf}{cf}
\newcommand{\gaS}{\operatorname{ga-S}}
\DeclareMathOperator{\LS}{LS}
\newcommand{\T}{\mathcal{T}}
\newcommand{\C}{\mathfrak{C}}

\journal{Annals of Pure and Applied Logic}

\begin{document}

\begin{frontmatter}



\title{Symmetry and the Union of Saturated Models in Superstable Abstract Elementary Classes}

\author{M.M.~VanDieren\corref{cor1}}
\cortext[cor1]{Corresponding Author}
\address{Department of Mathematics\\
Robert Morris University\\
Moon Township, PA, 15108, USA}
\ead{vandieren@rmu.edu}

\begin{abstract}

Our main result  (Theorem \ref{main theorem}) suggests a possible dividing line ($\mu$-superstable $+$ $\mu$-symmetric) for abstract elementary classes  without using extra set-theoretic assumptions or tameness.  This theorem illuminates the structural side of such a dividing line.  
 
 \begin{theorem}\label{main theorem}
  Let $\K$ be an abstract elementary class with no maximal models of cardinality $\mu^+$ which satisfies the joint embedding and amalgamation properties.  Suppose $\mu\geq\LS(\K)$. If $\K$ is $\mu$- and $\mu^+$-superstable  and satisfies $\mu^+$-symmetry, then for any increasing  sequence $\langle M_i\in\K_{\geq\mu^{+}}\mid i<\theta<(\sup\|M_i\|)^+\rangle$  of $\mu^+$-saturated models, $\Union_{i<\theta}M_i$ is $\mu^+$-saturated.
 
 \end{theorem}
We also apply results of \cite{Va3-SS} and use towers  to transfer symmetry from $\mu^+$ down to $\mu$ in abstract elementary classes which are both $\mu$- and $\mu^+$-superstable:  
 \begin{theorem}\label{symmetry transfer}
Suppose $\K$ is an abstract elementary class satisfying the amalgamation and joint embedding properties and that $\K$
is both $\mu$- and $\mu^+$-superstable.  If $\K$ has symmetry for non-$\mu^+$-splitting, then $\K$ has symmetry for non-$\mu$-splitting.
\end{theorem}

\end{abstract}

\begin{keyword}
saturated models \sep abstract elementary classes \sep superstability  \sep splitting \sep limit models

\MSC[2008] 03C48 \sep 03C45 \sep 03C50

\end{keyword}

\end{frontmatter}

In first-order logic, the statement,  the union of 
any increasing sequence $\langle M_i\mid i<\theta\rangle$  of saturated models is saturated, is a consequence  of superstability (\cite{Ha} and \cite[Theorem III.3.11]{She}).  In fact, the converse is also true \cite{AG}.  Our paper provides a new first-order proof of Theorem III.3.11 of \cite{She} when $\kappa(T)=\aleph_0$.

In abstract elementary classes (AECs), there are several approaches to generalizing superstability, and there is not yet a consensus on the correct notion.  In fact it could be that superstability breaks down into several distinct dividing lines.  Shelah suggests the existence of superlimits of every sufficiently large cardinality \cite[Chapter N Section 2]{Sh-AECbook} as the definition of superstability. Elsewhere he uses frames, but in his categoricity transfer results (e.g. \cite{Sh394}) he makes use of a localized notion more similar to $\mu$-superstability (Definition \ref{ss assm}).  

In this paper we examine how
 the statement,  that  the union of 
any increasing sequence $\langle M_i\mid i<\theta\rangle$  of saturated models is saturated, and $\mu$-superstability interact in abstract elementary classes.  
 
 There has been much progress in understanding the interaction.  We refer the reader to the introduction of \cite{BoVa-tame} for an extensive review of the history of the union of saturated models and the various proposals for a definition of superstability in AECs.  The most general result to date is due to Boney and Vasey for tame AECs.  They prove that a version of superstability  and tameness imply that the union of an increasing chain of $\mu$-saturated models is $\mu$-saturated for  $\mu>\beth_\lambda=\lambda>\LS(\K)$  \cite[Theorem 0.1]{BoVa-tame}.
  
  We prove a related result here.  Our result differs from \cite{BoVa-tame} in both assumptions and methodology.  We do not assume tameness, nor the existence of arbitrarily large models, and  $\mu$ does not need to be large.  Our methods involve limit models (and implicitly towers) and non-splitting instead of the machinery of averages and forking.  Additionally our proof is shorter.

 Underlying the proof of Theorem \ref{main theorem} are towers.
A tower is a relatively new model-theoretic concept unique to abstract elementary classes.   Towers were introduced by Shelah and Villaveces \cite{ShVi} as a tool to prove the uniqueness of limit models and later used by VanDieren \cite{Va1}, \cite{Va2} and by Grossberg, VanDieren, and Villaveces \cite{GVV}. 

\begin{definition}
A \emph{tower} is a sequence of length $\alpha$ of  limit models, denoted by $\bar M=\langle M_i\in\K_\mu\mid i<\alpha\rangle$, along with a sequence of designated elements $\bar a=\langle a_{i}\in M_{i+1}\backslash M_i\mid i+1<\alpha\rangle$ and a sequence of designated submodels $\bar N=\langle N_{i}\mid i+1<\alpha\rangle$ for which
 $M_i\prec_{\K}M_{i+1}$, $\tp(a_i/M_i)$ does not $\mu$-split over $N_i$, and $M_i$ is universal over $N_i$ (see for instance Definition I.5.1 of \cite{Va1}). 
 \end{definition}

Unlike many of the model-theoretic concepts in the literature of abstract elementary classes, the concept of a tower does not have a pre-established first-order analog.  Therefore there is a need to understand the applications and limitations of this concept.  
In \cite{Va3-SS}, VanDieren establishes that the statement that reduced towers are continuous is equivalent to symmetry for $\mu$-superstable  abstract elementary classes (see Fact \ref{symmetry theorem}).  Here we further explore the connection between reduced towers and symmetry  by using reduced towers in the proof of Theorem \ref{symmetry transfer}.

We can use Theorem \ref{symmetry transfer} to weaken the assumptions of Corollary 1 of \cite{Va3-SS} by replacing categoricity in $\mu^+$ with categoricity in $\mu^{+n}$ for some $n<\omega$ to conclude symmetry for non-$\mu$-splitting (see Corollary \ref{categoricity corollary} in Section \ref{downward section}).  Additionally, we make progress on improving the work of  \cite{ShVi}, \cite{Va1}, \cite{Va2}, \cite{GVV}, and \cite{Va3-SS} by proving the uniqueness of limit models of cardinality $\mu$ follows from categoricity in $\mu^{+n}$ for some $n<\omega$ without requiring tameness.  The uniqueness of limit models has been explored by others, assuming  tameness (e.g. \cite{BoGr}).

On its own, transferring symmetry is an interesting property that has been studied by others.  For instance, Shelah and separately Boney and Vasey   
transfer symmetry in a frame between cardinals  under set-theoretic assumptions \cite[Section II]{She} or using some level of tameness \cite[Section 6]{BoVa-tame}, respectively.  Our paper differs from this work in a few ways.  First, we do not assume tameness nor set-theoretic assumptions, and we do not work within the full strength of a frame.  The methods of this paper include reduced towers whereas the other authors use the order property as one of many mechanisms to transfer symmetry.  This line of work  is further extended in \cite{VV}.

One of the main questions surrounding this work is the interaction between the hypothesis of $\mu$-superstability, $\mu$-symmetry, the uniqueness of limit models of cardinality $\mu$, and the statement that the union of an increasing chain of $\mu$-saturated models is $\mu$-saturated.
Theorem \ref{main theorem} compliments \cite{Va3-SS} where the statement, that the union of an increasing  sequence $\langle M_i\in\K_{\mu^+}\mid i<\theta\rangle$  of saturated models is saturated, implies $\mu$-symmetry.  The following  combination  of Theorem 4 and Theorem 5 of \cite{Va3-SS} is close to, but not, the converse of Theorem \ref{main theorem}.

\begin{theorem}\label{transfer theorem}

 Let $\K$ be an abstract elementary class satisfying the amalgamation and joint embedding properties.  Suppose $\K$ is $\mu$- and $\mu^+$-superstable.  
 If, in addition, $\K$ satisfies the  property that the union of any chain of saturated models of cardinality $\mu^+$ is saturated, then $\K$ has $\mu$-symmetry.
\end{theorem}
 
 In fact combining the results from \cite{Va3-SS} with the work here we get the implications depicted in Figure \ref{fig:ss}.
 
 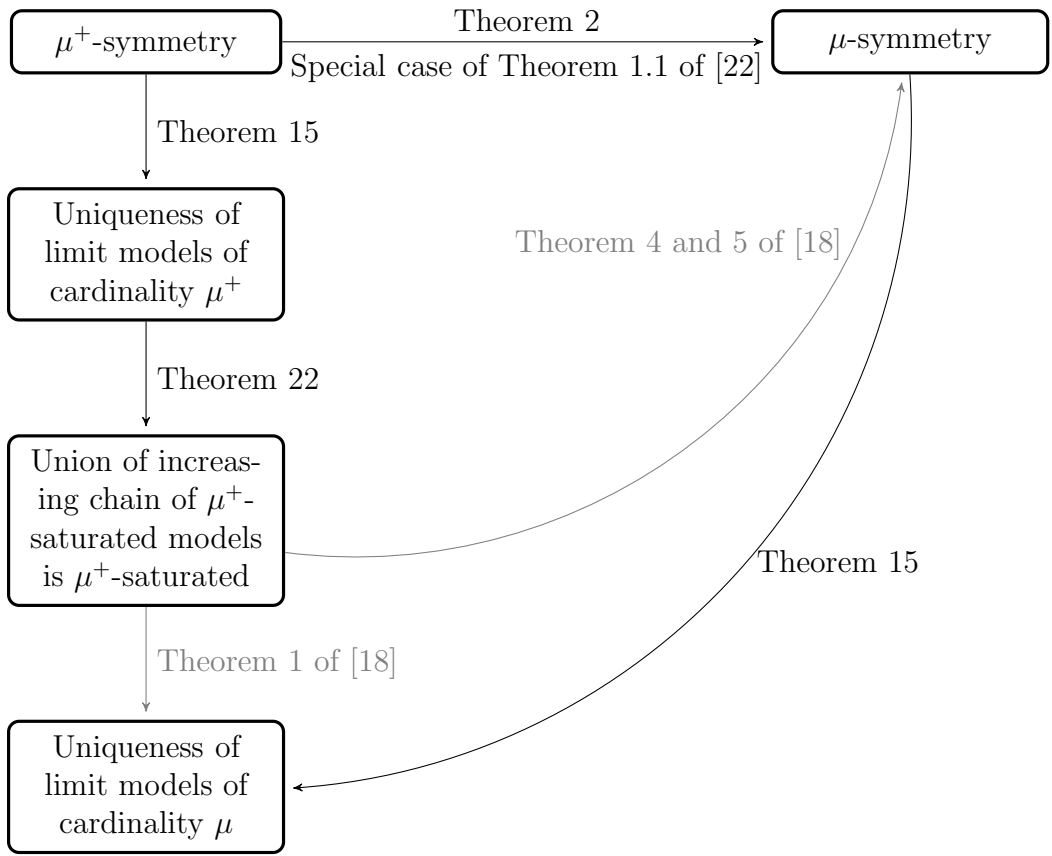
\begin{figure}[h]
\begin{tikzpicture}
 \node[punkt] (ss) {$\mu^+$-symmetry};
 \node[punkt, inner sep=5pt,below=1.5cm of ss]
 (unique) {Uniqueness of limit models of cardinality $\mu^+$};
  \node[punkt, inner sep=5pt, right=6.5cm of ss] (mu-sym) {$\mu$-symmetry};
   \draw [->, shorten >=3pt] (ss) to node[pos=0.5,above] {Theorem \ref{symmetry transfer}} node[pos=0.5, below]{Special case of Theorem 1.1 of \cite{VV}} (mu-sym);
 \draw [->, shorten >=3pt] (ss) to node[pos=0.5,right] {Theorem \ref{uniqueness thm}}(unique);
 \node[punkt, inner sep=5pt, below=1.5cm of unique] (union) {Union of increasing chain of $\mu^+$-saturated models is $\mu^+$-saturated};
 \node[punkt, inner sep=5pt, below=1.5cm of union] (uniqueness) {Uniqueness of limit models of cardinality $\mu$};
  \draw [->, shorten >=3pt] (unique) to node[pos=0.5,right] {Theorem \ref{saturated theorem}}(union);
    \draw [->, shorten >=3pt,gray] (union) to node[pos=0.5,right]{Theorem 1 of \cite{Va3-SS}}(uniqueness);
    \draw[->, shorten >=3pt, bend left=45](mu-sym) to node[pos=0.5,right]{Theorem \ref{uniqueness thm}}(uniqueness);
    \draw[->, shorten >=3pt, bend right=45,gray] (union) to node[pos=0.8,left]{Theorem 4 and 5 of \cite{Va3-SS}} (mu-sym);
   \end{tikzpicture}
   \caption{A diagram of some the local implications for  abstract elementary classes which are both $\mu$ and $\mu^+$-superstable, satisfy the amalgamation and joint embedding properties, and have no maximal models of carnality $\mu^+$.  The gray labels represent results from \cite{Va3-SS}.} \label{fig:ss}
\end{figure}

 This diagram suggests several questions including: does the uniqueness of limit models of cardinality $\mu$ imply $\mu^+$-symmetry (or even $\mu$-symmetry) in $\mu$-superstable classes?
 There are also many questions that remain open concerning the non-structure side of any of the proposed definitions for superstability for AECs.  In fact, very little is known about the implications of the failure of $\mu$-superstability.  However VanDieren and Vasey have shown that with  $\mu$-superstability holding in sufficiently many cardinals, failure of $\mu$-symmetry would imply the order property \cite{VV2}, which Shelah has claimed implies many models \cite{Sh394}.

The paper is structured as follows.  Section \ref{sec:background} provides some of the pre-requisite material.  
The subsequent section contains an observation about how saturated models and limit models are related which is key in being able to construct towers of cardinality $\mu^+$ from towers of cardinality $\mu$.  This construction is the basis for the proof of Theorem \ref{symmetry transfer} which appears in
Section \ref{downward section}.
Then in Section \ref{sec:warm-up} we prove a weaker result than Theorem \ref{main theorem} to highlight the structure of the proof of Theorem \ref{main theorem} since the construction in the proof of Theorem \ref{main theorem} is more complicated requiring a directed system instead of an  increasing chain.
Finally, in Section \ref{sec:main theorem} we prove Theorem \ref{main theorem}.   We finish the paper with a summary of how this work fits into the recently growing body of research on superstability in abstract elementary classes.

At the suggestion of the referees, this paper is the synthesis of two preprints \cite{Vold-union} and \cite{Vold-transfer} which were disseminated in July of 2015.

\section{Background}\label{sec:background}

For the remainder of this paper we will assume that $\K$ is an abstract elementary class with no maximal models of cardinality $\mu^+$ satisfying the joint embedding and amalgamation properties.

 Many of the pre-requisite definitions and notation can be found in \cite{GVV}.  Here we recall the more specialized concepts that we will be using explicitly in the proofs of Theorem \ref{main theorem} and Theorem \ref{symmetry transfer}.

We will use the following definition of $\mu$-superstability:
\begin{definition}\label{ss assm}
$\K$ is \emph{$\mu$-superstable} if $\K$ is Galois-stable in $\mu$ and 
 $\mu$-splitting satisfies the property:
for all infinite $\alpha$, for every sequence $\langle M_i\mid i<\alpha\rangle$ of
  limit models of cardinality $\mu$ with $M_{i+1}$ universal over $M_i$, and for every $p\in\gaS(M_\alpha)$, where
  $M_\alpha=\bigcup_{i<\alpha}M_i$, we have that
there exists $i<\alpha$ such that $p$
does not $\mu$-split over $M_i$.

\end{definition}

\begin{remark}
Other definitions of $\mu$-superstability for AECs appear in the literature. For instance Vasey introduces a very similar definition of superstability with the additional requirement of no maximal models of cardinality $\mu$  \cite[Definition 10.1]{V1}. We choose to separate this condition out to be consistent with the presentation in \cite{GVV}, \cite{Va3-SS}, etc.


\end{remark}

In \cite[Chapter N Section 2]{Sh-AECbook}, Shelah discusses the problem of generalizing first-order superstability to AECs. There Shelah suggests using the existence of a superlimit model in every sufficiently large cardinality as a dividing line.  Here we take a different, more local approach where an AEC may exhibit superstable-like properties in small cardinalities but not necessarily in larger cardinalities.  This helps to classify, for instance, those classes such as the Hart-Shelah example \cite{HS} which have structural properties in small cardinalities but non-structural attributes in larger cardinalities.
In \cite{VV} and \cite{VV2} we consider how Theorem \ref{main theorem} and Theorem \ref{symmetry transfer} color the global picture of superstability when one assumes categoricity or tameness.
Guided by the first-order characterization of superstability that the union of an increasing chain of saturated models is saturated,
Theorem \ref{main theorem} provides evidence that Definition \ref{ss assm} along with $\mu$-symmetry may be a reasonable generalization of superstability.

\begin{definition}\label{sym defn}
We say that an abstract elementary class  exhibits \emph{symmetry for non-$\mu$-splitting} if  whenever models $M,M_0,N\in\K_\mu$ and elements $a$ and $b$  satisfy the conditions \ref{limit sym cond}-\ref{last} below, then there exists  $M^b$  a limit model over $M_0$, containing $b$, so that $\tp(a/M^b)$ does not $\mu$-split over $N$.  See Figure \ref{fig:sym}.  We will abbreviate this concept by \emph{$\mu$-symmetry} when it is clear that the dependence relation is $\mu$-splitting. 
\begin{enumerate} 
\item\label{limit sym cond} $M$ is universal over $M_0$ and $M_0$ is a limit model over $N$.
\item\label{a cond}  $a\in M\backslash M_0$.
\item\label{a non-split} $\tp(a/M_0)$ is non-algebraic and does not $\mu$-split over $N$.
\item\label{last} $\tp(b/M)$ is non-algebraic and does not $\mu$-split over $M_0$. 
   
\end{enumerate}

\end{definition}

\begin{figure}[h]
\begin{tikzpicture}[rounded corners=5mm, scale=3,inner sep=.5mm]
\draw (0,1.25) rectangle (.75,.5);
\draw (.25,.75) node {$N$};
\draw (0,0) rectangle (3,1.25);
\draw (0,1.25) rectangle (1,0);
\draw (.85,.25) node {$M_0$};
\draw (3.2, .25) node {$M$};
\draw[color=gray] (0,1.25) rectangle (1.5, -.5);
\node at (1.1,-.25)[circle, fill, draw, label=45:$b$] {};
\node at (2,.75)[circle, fill, draw, label=45:$a$] {};
\draw[color=gray] (1.75,-.25) node {$M^{b}$};
\end{tikzpicture}
\caption{A diagram of the models and elements in the definition of symmetry. We assume the type $\tp(b/M)$ does not $\mu$-split over $M_0$ and $\tp(a/M_0)$ does not $\mu$-split over $N$.  Symmetry implies the existence of $M^b$ a limit model over $M_0$ containing $b$, so that $\tp(a/M^b)$  does not $\mu$-split over $N$.} \label{fig:sym}
\end{figure}
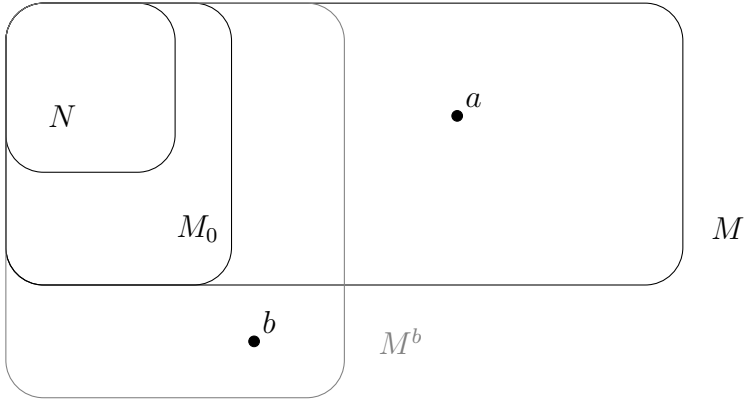

This concept of $\mu$-symmetry was introduced in \cite{Va3-SS} and shown to be equivalent to a property about reduced towers (see Fact \ref{symmetry theorem}).  Before stating this result,
let us recall a bit of terminology regarding towers.  The collection of all towers $(\bar M,\bar a,\bar N)$ made up of models of cardinality $\mu$ and sequences indexed by $\alpha$ is denoted by $\K^*_{\mu,\alpha}$.  For $(\bar M,\bar a,\bar N)\in\K^*_{\mu,\alpha}$, if $\beta<\alpha$ then we write $(\bar M,\bar a,\bar N)\restriction\beta$ for the tower made of the subseqences $\bar M\restriction\beta=\langle M_i\mid i<\beta\rangle$, $\bar a\restriction\beta=\langle a_i\mid i+1<\beta\rangle$, and $\bar N\restriction\beta=\langle N_i\mid i+1<\beta\rangle$.  We sometimes abbreviate the tower $(\bar M,\bar a,\bar N)$ by $\T$.

\begin{definition}

For towers $(\bar M,\bar a,\bar N)$ and $(\bar M',\bar a',\bar N')$ in $\K^*_{\mu,\alpha}$, we say $$(\bar M,\bar a,\bar N)\leq (\bar M',\bar a',\bar N')$$ if for all $i<\alpha$, $M_i\preceq_{\K}M'_i$, $\bar a=\bar a'$, $\bar N=\bar N'$ and whenever $M'_i$ is a proper extension of $M_i$, then $M'_i$ is universal over $M_i$.  If for each $i<\alpha$,  $M'_i $ is universal over $M_i$ we will write $(\bar M,\bar a,\bar N)< (\bar M',\bar a',\bar N')$.
\end{definition}

\begin{definition}\label{reduced defn}\index{reduced towers}
A tower $(\bar M,\bar a,\bar N)\in\K^*_{\mu,\alpha}$ is said to 
be \emph{reduced} provided that for every $(\bar M',\bar a,\bar
N)\in\K^*_{\mu,\alpha}$ with
$(\bar M,\bar a,\bar N)\leq(\bar M',\bar a,\bar
N)$ we have that for every
$i<\alpha$,
$$(*)_i\quad M'_i\cap\Union_{j<\alpha}M_j = M_i.$$
\end{definition}

The following result from  \cite{Va3-SS} links together symmetry and reduced towers:
\begin{fact}\label{symmetry theorem}
Assume $\K$ is an abstract elementary class satisfying superstability properties for $\mu$.  Then the following are equivalent:
\begin{enumerate}
\item\label{sym} $\K$ has symmetry for non-$\mu$-splitting.
\item\label{red} If $(\bar M,\bar a,\bar N)\in\K^*_{\mu,\alpha}$ is a reduced tower, then $\bar M$ is a continuous sequence (i.e. for every limit ordinal $\beta<\alpha$, we have $M_\beta=\Union_{i<\beta}M_i$).

\end{enumerate}
\end{fact}

There are a few facts about reduced towers known to hold under the assumption of $\mu$-superstability. 
The following 
appears in \cite{Va1} as Theorem III.11.2.

\begin{fact}[Density of reduced towers]\label{density of reduced}\index{reduced towers!density of} Suppose $\K$ is an abstract elementary class satisfying the joint embedding and amalgamation properties.  If $\K$ is $\mu$-superstable, then there exists a reduced
$<$-extension of every tower in
$\K^*_{\mu,\alpha}$.
\end{fact}

The next lemma is Lemma III.11.5 in \cite{Va1}.
\begin{fact}\label{monotonicity}
Suppose $\K$ is a $\mu$-superstable abstract elementary class satisfying the joint embedding and amalgamation properties. 
Suppose that $(\bar M,\bar a,\bar N)\in\K^*_{\mu,\alpha}$ is
reduced.  If  $\beta<\alpha$, then $(\bar M,\bar
a,\bar N)\restriction \beta$ is reduced.
\end{fact}

Before moving onto the proofs of Theorem \ref{main theorem} and Theorem \ref{symmetry transfer}, we state a fact about direct limits that we will use in Section \ref{sec:main theorem}.  It is implicit in the proof of Lemma 2.12 of \cite{GV}.

\begin{fact}\label{direct limit lemma}
Suppose that $\theta$ is a limit ordinal and $\langle M_i\in\K_{\mu}\mid i<\theta\rangle$ and $\langle f_{i,j}\mid i\leq j<\theta\rangle$ form a directed system.  If $\langle N_i\mid i<\theta\rangle$ is an increasing and continuous sequence of models so that for every $i<\theta$, $N_i\prec_{\K}M_i$ and $f_{i,i+1}\restriction N_i=\id_{N_i}$, then there is a direct limit $M^*$ of the system  and $\K$-embeddings $\langle f_{i,\theta}\mid i<\theta\rangle$  so that $\Union_{i<\theta}N_i\preceq_{\K}M^*$ and $f_{i,\theta}\restriction N_i=id_{N_i}$.

\end{fact}

\section{Limit and Saturated Models}\label{sec:limit and sat}

In this section we establish that for $\mu$-superstable and $\mu$-symmetric abstract elementary classes, limit models are in fact saturated.
  We begin by noticing that a
 $(\mu,\mu^+)$-limit model\footnote{$M$ is a $(\mu,\mu^+)$-limit model if $M=\Union_{i<\mu^+}M_i$ for some increasing and continuous sequence of models $\langle M_i\in\K_{\mu}\mid i<\mu^+\rangle$ where $M_{i+1}$ is universal over $M_i$ for each $i<\mu^+$.} is isomorphic to a $(\mu^+,\mu^+)$-limit model in $\mu^+$-stable abstract elementary classes.

\begin{proposition}\label{mu-plus-limit}
If $\K$ is $\mu^+$-stable and does not have a maximal model of cardinality $\mu^+$, then any $(\mu,\mu^+)$-limit model is a $(\mu^+,\mu^+)$-limit model.
\end{proposition}
\begin{proof}
Let $M_{\mu^+}$ be a $(\mu,\mu^+)$-limit model witnessed by $\langle M_i\mid i\leq\mu^+\rangle$.  Without loss of generality $M_{i+1}$ is a $(\mu,\omega)$-limit over $M_i$.
By $\mu^+$-stability, we can fix $N$ a $(\mu^+,\mu^+)$-limit model witnessed by $\langle N'_i\mid i\leq\mu^+\rangle$ so that $M_0\prec_{\K}N_0$.  Fix $\{a_i\mid i<\mu^+\}$ to be an enumeration of $N$.  We will define an increasing and continuous sequence $\langle f_i\mid i\leq\mu^+\rangle$ so that
\begin{enumerate}
\item for $i<j\leq\mu^+$, $f_i=f_j\restriction M_i$
\item for $i\leq \mu^+$ limit, $f_i=\Union_{j<i}f_j\restriction M_j$
\item $f_i:M_i\rightarrow N$
\item $a_i\in\text{range}(f_{i+1}\restriction M_{i+1})$
\end{enumerate}

Take $f_0=\id$.  For $i$ limit, by the continuity of $\bar M$, we can take $f_i:=\Union_{j<\mu^+}f_j\restriction M_j$.  For the successor case $i=j+1$, fix $\grave f_j\in\Aut(\C)$ an extension of $f_j$.    
Let $k<\mu^+$ be such that  $a_j\in N_k$ and $N_k$ is universal over $\grave f_j(M_j)$.
 Let $M'_{j+1}\prec_{\K}\grave f^{-1}_j(N_k)$ be a $(\mu,\omega)$-limit over  $M_i$ containing $\grave f^{-1}_j(a_j)$.  This is possible since $N_k$ is universal over $\grave f_j(M_j)$.
  By the uniqueness of $(\mu,\omega)$-limit models, there exists $g:M'_{j+1}\cong_{M_j}M_{j+1}$.  Now take $f_{j+1}:=\grave f_j\circ g^{-1}\restriction M_{j+1}$.  Notice $f_{j+1}\restriction M_{j}=f_j\restriction M_j$ since $g^{-1}$ fixes $M_j$.  Also, by our choice of $g$, $a_j\in f_{j+1}[M_{j+1}]$ as required.  
  
  Notice that $f_{\mu^+}$ is an isomorphism between $M_{\mu^+}$ and $N$.

\end{proof}

The direct approach of constructing  a saturated model is to realize all the relevant types.  Another method is to show that the model is a limit model and depending on the context, there are times when limit models are saturated.  Trivially, a $(\mu,\mu^+)$-limit model is saturated.  Moreover, if the class $\K$ satisfies the condition 
 \begin{quote}
for every  $l\in\{1,2\}$, and every pair of limit ordinals $\theta_l<\mu^+$, and pair of  $(\mu,\theta_l)$-limit models $M_l$,
we have $M_1\cong M_2$, 
 \end{quote}
 then any limit model of cardinality $\mu$ is also saturated.  To see this, suppose $M$ is a $(\mu,\theta)$-limit model and fix  $\chi<\mu$ and $N\in\K_\chi$ with $N\prec_{\K}M$.  By uniqueness of limit models, we can think of  $M$ as $(\mu,\chi^+)$-limit model witnessed by $\langle M_i\mid i<\chi^+\rangle$.  The model $N$ appears in one of the $M_i$, so $M_{i+1}$ will realize all the types over $M_i$, and hence over $N$.
  
 In our context, under the hypothesis of Theorem \ref{main theorem}, we have uniqueness of limit models of cardinality $\mu^+$:

 \begin{theorem}\label{uniqueness thm}
 Let $\K$ be an abstract elementary class which satisfies the joint embedding and amalgamation properties.  
 Suppose $\mu$ is a cardinal $\geq\LS(\K)$ and $\theta_1$ and $\theta_2$ are limit ordinals $<\mu^+$.
 If $\K$ is $\mu$-superstable and satisfies $\mu$-symmetry, then for $M_1$ and $M_2$ which are $(\mu, \theta_1)$ and $(\mu,\theta_2)$-limit models over $N$, respectively, we have that 
 $M_1$ is isomorphic to $M_2$ over $N$.  Moreover the limit model of cardinality $\mu$ is saturated.
 \end{theorem}
\begin{proof}
 This is just a restatement of Theorem 5 of \cite{Va3-SS} and the proof of Theorem 1.9 of \cite{GVV}.
\end{proof}

 Combining Theorem \ref{uniqueness thm} with Proposition \ref{mu-plus-limit}, we get the following corollary.  
  \begin{corollary}\label{limit is sat}
 Let $\K$ be an abstract elementary class which satisfies the joint embedding and amalgamation properties.  
 Suppose $\kappa$ is a cardinal $\geq\LS(\K)$,   and $\theta$ is  limit ordinal $<\kappa^{++}$.
 
 If $\K$ is $\kappa$-stable, $\kappa^+$-superstable and satisfies $\kappa^+$-symmetry, then any $(\kappa^+,\theta)$-limit model is also a $(\kappa,\kappa^+)$-limit model.
 \end{corollary}

\section{Downward Symmetry Transfer}\label{downward section}

In this section we provide the proof of Theorem \ref{symmetry transfer}.  While the result follows from Theorem 4 and 5 of \cite{Va3-SS}, we include the proof here for completeness since  \cite{Va3-SS} is currently under review and has not yet been published.   Additionally, the proof of Theorem \ref{symmetry transfer} serves as the blueprint for the successor step
for a more general result of transferring symmetry downward that appears in the unpublished work \cite{VV}.

In the proof of Theorem \ref{symmetry transfer}, we will be using towers composed of models of cardinality $\mu$ and other towers composed of models of cardinality $\mu^+$.  These towers will  be based on the same sequence of elements $\langle a_\beta\mid \beta<\delta\rangle$.  To distinguish the towers of 
models of size $\mu^+$ from those of size $\mu$, we will use different notation.  The models of cardinality $\mu^+$ will be decorated with an asterisk ($*$), accent ($\grave{}$), or a $\mu^+$ in the superscript.  All other models in this proof will have cardinality  $\mu$.

\begin{proof}[Proof of Theorem \ref{symmetry transfer}]
Suppose $\K$ does not have symmetry for $\mu$-non-splitting.  By Fact \ref{symmetry theorem} and the $\mu$-superstability assumption, $\K$ has a reduced discontinuous tower.  Let $\alpha$ be the minimal ordinal such that $\K$ has a reduced discontinuous tower of length $\alpha$.  By Fact \ref{monotonicity}, we may assume that $\alpha=\delta+1$ for some limit ordinal $\delta$.  Fix $\T=(\bar M,\bar a,\bar N)\in\K^*_{\mu,\alpha}$ a reduced discontinuous tower  with $b\in M_\delta\backslash \Union_{\beta<\delta}M_\beta$.  By Fact \ref{density of reduced} and minimality of $\alpha$, we can build an increasing and continuous chain of reduced, continuous towers $\langle\T^i\mid i<\mu^+\rangle$ extending $\T\restriction\delta$.  

For each $\beta<\delta$, set $M^{\mu^+}_\beta:=\Union_{i<\mu^+}M^i_\beta$.  
Notice that for each $\beta<\delta$ 
\begin{equation}\label{non-split}
\tp(a_\beta/M^{\mu^+}_\beta)\text{ does not }\mu\text{-split over }N_\beta.
\end{equation}
 If $\tp(a_\beta/M^{\mu^+}_\beta)$ did $\mu$-split over $N_\beta$, it would be witnessed by models inside some $M^i_\beta$, contradicting the fact that $\tp(a_\beta/M^{i}_\beta)$ does not $\mu$-split over $N_\beta$.

We will construct a tower in $\K^*_{\mu^+,\delta}$ from $\bar M^{\mu^+}$.    Notice that by construction, each $M^{\mu^+}_\beta$ is a $(\mu,\mu^+)$-limit model.  Therefore by Proposition \ref{mu-plus-limit}, each $M^{\mu^+}_\beta$ is a $(\mu^+,\mu^+)$-limit model.  Fix $\langle \grave M^i_\beta\mid i<\mu^+\rangle$ witnessing that $M^{\mu^+}_\beta$ is a $(\mu^+,\mu^+)$-limit model.  Without loss of generality we can assume that $N_\beta\prec_{\K}\grave M^0_\beta$.  
By $\mu^+$-superstability we know that for each $\beta<\delta$ there is $i(\beta)<\mu^+$ so that $\tp(a_\beta/M^{\mu^+}_\beta)$ does not $\mu^+$-split over $\grave M^{i(\beta)}_\beta$.  Set $N^{\mu^+}_\beta:=\grave M^{i(\beta)}_\beta$.  Notice that $(\bar M^{\mu^+},\bar a,\bar N^{\mu^+})$ is a tower in $\K^*_{\mu^+,\delta}$.  Extend $(\bar M^{\mu^+},\bar a,\bar N^{\mu^+})$ to a tower $\T^{\mu^+}\in\K^*_{\mu^+,\alpha}$ by appending to $\bar M^{\mu^+}$ a $\mu^+$-limit model universal over $M_\delta$ which contains $\Union_{\beta<\delta}M^{\mu^+}_\beta$.  
Since $\T^{\mu^+}$ is discontinuous, by Fact \ref{symmetry theorem} and our $\mu^+$-symmetry assumption, we know that it is not reduced.

However, by our $\mu^+$-symmetry assumption, Fact \ref{symmetry theorem} and Fact \ref{density of reduced} imply that there exists a reduced, continuous tower $\T^*\in\K^*_{\mu^+,\alpha}$ extending $\T^{\mu^+}$.  By multiple applications of Fact \ref{density of reduced}, we may assume that in $\T^*$ each $M^*_\beta$ is a $(\mu^+,\mu^+)$-limit over $M^{\mu^+}_\beta$.  See Fig. \ref{fig:tower}.
\begin{figure}[h]
\begin{tikzpicture}[rounded corners=5mm,scale =2.9,inner sep=.5mm]
\draw (0,1.5) rectangle (.75,.5);
\draw (0,1.5) rectangle (1.75,1);
\draw (.25,.65) node {$N_0$};
\draw (1.25,1.1) node {$N_\beta$};
\draw[rounded corners=5mm]  (0, 1.5) --(0,.75)-- (1,-1) -- (1.5,-1)  -- (1.75,1)--(1.75,1.5)--  cycle;
\draw (1.85,.7) node {$N^{\mu^+}_\beta$};
\draw (0,0) rectangle (4,1.5);
\draw (.85,.25) node {$M_0$};
\draw(1.4,.25) node {$M_1$};
\draw (1.8,.25) node {$\dots M_\beta$};
\draw (2.35,.25) node {$M_{\beta+1}$};
\draw (3.15,.2) node {$\dots\displaystyle{\Union_{k<\delta}M_k}$};
\draw (3.85, .25) node {$M_\delta$};
\draw (-.5,.25) node {$(\bar M,\bar a,\bar N)$};
\draw (0,1.5) rectangle (3.5, -.4);
\draw[rounded corners=5mm]  (0, 1.5) -- (0,-2) -- (3.6,-2)  -- (4,0)--(4,1.5) --  cycle;
\draw[rounded corners=5mm]  (0, 1.5) -- (0,-1.35) -- (3.5,-1.35)  -- (4,0)--(4,1.5) --  cycle;
\draw (.85,-.15) node {$M^{i}_0$};
\draw (1.8,-.15) node {$\dots M^{i}_\beta$};
\draw (2.35,-.15) node {$M^{i}_{\beta+1}$};
\draw(1.4,-.15) node {$M^{i}_1$};
\draw (3.15,-.2) node {$\dots\displaystyle{\Union_{l<\delta}M^{i}_l}$};
\draw (-.5,-.15) node {$\T^i$};
\draw (.85,-.6) node {$\vdots$};
\draw (1.75,-.6) node {$\vdots$};
\draw (2.35,-.6) node {$\vdots$};
\draw (3.2,-.6) node {$\vdots$};
\draw (1.35,-.6) node {$\vdots$};
\draw (0,1.5) rectangle (3.5, -1.35);
\draw (0,1.5) rectangle (1,-2);
\draw(0,1.5) rectangle (1.5, -2);
\draw (0,1.5) rectangle (2.5, -2);
\draw (0,1.5) rectangle (2,-2);
\draw (.8,-1.15) node {$M^{\mu^+}_0$};
\draw (1.8,-1.15) node {$ M^{\mu^+}_\beta$};
\draw (2.3,-1.15) node {$M^{\mu^+}_{\beta+1}$};
\draw(1.35,-1.15) node {$M^{\mu^+}_1$};
\draw (3.1,-1.2) node {$\dots\displaystyle{\Union_{l<\delta}M^{\mu^+}_l}$};
\draw (-.5,-1.15) node {$\T^{\mu^+}$};
\draw (-.5,-1.75) node {$\T^{*}$};
\draw (.8,-1.75) node {$M^{*}_0$};
\draw(1.4,-1.75) node {$M^*_1$};
\draw (1.8,-1.75) node {$ M^{*}_\beta$};
\draw (2.3,-1.75) node {$M^{*}_{\beta+1}$};
\node at (3.75,.75)[circle, fill, draw, label=90:$b$] {};
\node at (2.25,.65)[circle, fill, draw, label=290:$a_\beta$] {};
\node at (1.1,.65)[circle, fill, draw, label=290:$a_1$] {};
\draw (3.65, -.6) node {$M^{\mu^+}_\delta$};
\draw (3.1,-1.8) node {$\dots\displaystyle{\Union_{\beta<\delta}M^*_\beta}=M^*_\delta$};
\end{tikzpicture}
\caption{The  towers in the proof of Theorem \ref{symmetry transfer}} \label{fig:tower}
\end{figure}
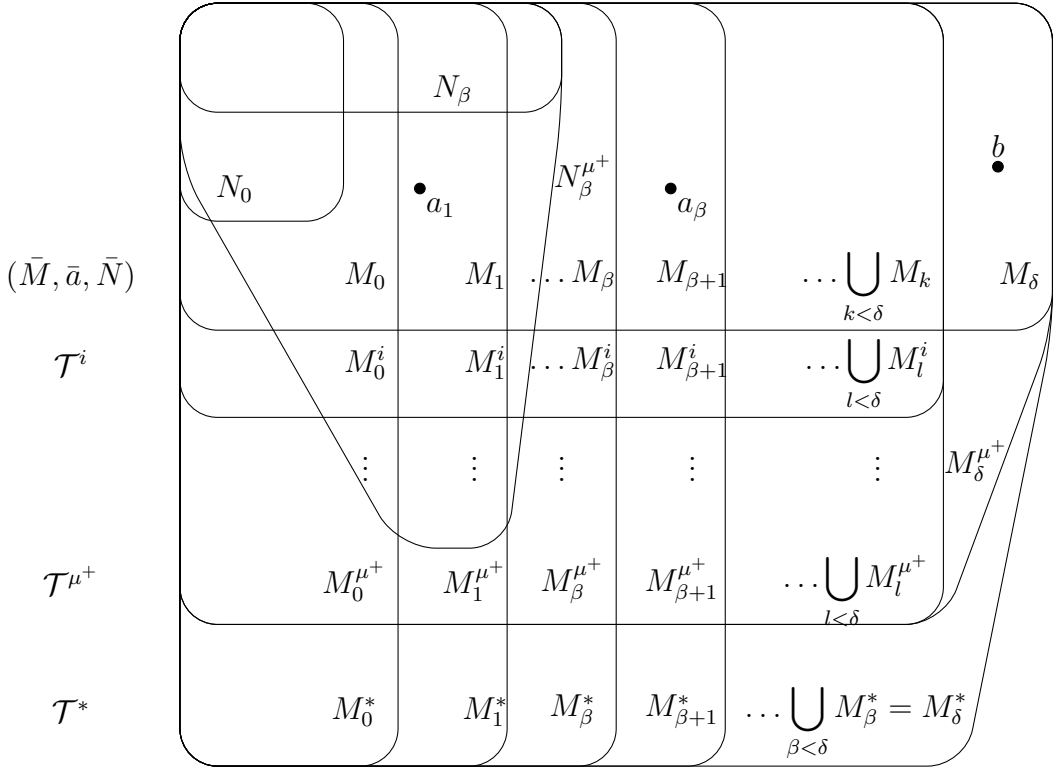

\begin{claim}\label{star non-split}
For every $\beta<\alpha$, 
$\tp(a_\beta/M^*_\beta)$ does not $\mu$-split over $N_\beta$.
\end{claim}
\begin{proof}
Since $M^*_\beta$ and $M^{\mu^+}_\beta$ are both $(\mu^+,\mu^+)$-limit models over $N^{\mu^+}_\beta$, there exists $f:M^*_\beta\cong_{N^{\mu^+}_\beta}M^{\mu^+}_\beta$.  Since $\T^*$ is a tower extending $\T^{\mu^+}$, we know that $\tp(a_\beta/M^*_\beta)$ does not $\mu^+$-split over $N^{\mu^+}_\beta$.  Therefore by the definition of non-splitting, it must be the case that $\tp(f(a_\beta)/M^{\mu^+}_\beta)=\tp(a_\beta/M^{\mu^+}_\beta)$.  From this equality of types we can fix $g\in\Aut_{M^{\mu^+}_\beta}(\C)$ with $g(f(a_\beta))=a_\beta$.
An application of $(g\circ f)^{-1}$ to $(\ref{non-split})$ yields the statement of the claim.

\end{proof}

Since $\T^*$ is continuous and extends $\T^{\mu^+}$ which contains $b$, there is $\beta<\delta$ such that $b\in M^*_\beta$.  Fix such a $\beta$.

We now will define a tower $\T^b\in\K^*_{\mu,\alpha}$ extending $\T$.  For $\gamma<\beta$, take $M^b_\gamma:=M_\gamma$.  For $\gamma=\beta$, let $M^b_\gamma$ be a $(\mu,\mu)$-limit model over $M_\gamma$ inside $M^*_\gamma$ so that $b\in M^b_\gamma$.  For $\gamma>\beta$, take $M^b_\gamma$ to be a $(\mu,\mu)$-limit model over $M_\gamma$ so that $\Union_{\xi<\gamma}M^b_{\xi}\prec_{\K}M^b_\gamma$.  Notice that by Claim \ref{star non-split} and monotonicity of non-splitting, the tower $\T^b$ defined as $(\bar M^b,\bar a,\bar N)$ is a tower extending $\T$ with $b\in (M^b_\beta\backslash M_\beta)\bigcap M_\alpha$.  This contradicts our assumption that $\T$ was reduced.

\end{proof}

The following is a strengthening of Corollary 1 from \cite{Va3-SS}.  In particular, here we replace the assumption that $\K$ is categorical in $\mu^+$ with the statement: $\K$ is categorical in $\mu^{+n}$ for some $n<\omega$.
\begin{corollary}\label{categoricity corollary}
Suppose that $\K$ satisfies the amalgamation and joint embedding properties and has arbitrarily large models.  Fix $\mu$ a cardinal $\geq\LS(\K)$.
If $\K$ is categorical in $\lambda=\mu^{+n}$, then $\K$ has symmetry for non-$\mu$-splitting.

\end{corollary}

\begin{proof}
Notice that categoricity in $\lambda$ and the existence of arbitrarily large models allows us to make use of EM-models.  These assumptions imply stability in $\kappa$ for $\kappa=\mu^{+k}$ with $0\leq k<n$ (see for instance Theorem 8.2.1 of \cite{Ba}).
Also,  $\kappa$-superstability for $\kappa=\mu^{+k}$ for $0\leq k<n$ follows from categoricity by the argument of Theorem 2.2.1 of \cite{ShVi}.  While \cite{ShVi} uses the assumption of GCH, it can be eliminated here because we are assuming the amalgamation property \cite[Theorem 6.3]{GVas}.
By Corollary 1 of \cite{Va3-SS}, we get symmetry for non-$\mu^{+(n-1)}$-splitting.  Then, Theorem \ref{symmetry transfer} gives us symmetry for non-$\mu^{k}$-splitting for the remaining $0\leq k<n-1$.

\end{proof}

Using Corollary \ref{categoricity corollary}, we add to the line of work on the uniqueness of limit models 
by deriving a relative of the main result, Theorem 1.9, of \cite{GVV} and Theorem 1 of \cite{Va3-SS}.

\begin{corollary}
Suppose that $\K$ satisfies the amalgamation and joint embedding properties and has arbitrarily large models.  Fix $\mu$ a cardinal $\geq\LS(\K)$.  If $\K$ is categorical in $\mu^{+n}$, then 
for each $0<k<n$, and limit ordinals $\theta_1,\theta_2<\mu^{+(k+1)}$, if $M_1$ and $M_2$ are $(\mu^{+k},\theta_1)$- and $(\mu^{+k},\theta_2)$-limit models over $N$, respectively, then $M_1$ is isomorphic to $M_2$ over $M$.

\end{corollary}
\begin{proof}
This follows from Corollary \ref{categoricity corollary}, Fact \ref{symmetry theorem}, and the arguments of \cite{GVV} which show that superstability plus the statement that reduced towers are continuous is enough to get uniqueness of limit models in a given cardinality.

\end{proof}


 \section{Union of Saturated Models: warm-up}\label{sec:warm-up}
 
 The goal of this section is to prove  the following warm-up to Theorem \ref{main theorem}.

   \begin{theorem}\label{limit thm}
 Let $\K$ be an abstract elementary class which satisfies the joint embedding and amalgamation properties.  Suppose that $\lambda$ and $\mu$ are cardinals $\geq\LS(\K)$ with  $\lambda\geq\mu^{++}$ and that $\theta$ is a limit ordinal $<\lambda^+$.  If $\K$ is $\mu^+$-superstable  and satisfies $\mu^+$-symmetry, then for any increasing  sequence $\langle M_i\mid i<\theta\rangle$  of $\mu^{++}$-saturated models of cardinality $\lambda$, $M=\Union_{i<\theta}M_i$ is $\mu^+$-saturated.
 \end{theorem}
 
Notice that the statement of Theorem \ref{limit thm} differs from Theorem \ref{main theorem} in two ways.  The cardinality, $\lambda$, of the saturated models in the chain is greater than or equal to the level of saturation, $\mu^{++}$, of the models $M_i$. Also, the level of saturation that we get in the  union is only $\mu^+$.
 
 The proof of this theorem will prepare us for a similar construction used
 in the proof Theorem \ref{main theorem}  with the addition of a directed system.  Given $N\prec_{\K}\Union_{i<\theta}M_i$ of cardinality $\mu$,
 the structure of the proof is to construct  an increasing chain $\langle M^*_i\mid i<\theta\rangle$ of models of cardinality $\mu^+$ inside $\Union_{i<\theta}M_i$ 
 so that $M^*:=\Union_{i<\theta}M^*_i$ contains
$N$ and
 so that $M^*_{i+1}$ is universal over $M^*_i$.  Then by definition of limit models,  $M^*$ is a $(\mu^+,\theta)$-limit model. By Theorem \ref{uniqueness thm}, $M^*$ is saturated, and every type over $N$ is realized in $M^*$ and hence in $\Union_{i<\theta}M_i$.

 \begin{proof}
 First observe that we may assume that the sequence $\langle M_i\mid i<\theta\rangle$   is continuous.  Otherwise, we could consider  $\langle M_i\mid i<\theta\rangle$ a counter-example of the theorem of minimal length and proceed to prove the theorem by contradiction using the argument below.
 
Fix $N\in\K_\mu$ with $N\prec_{\K}M$  and $p\in \gaS(N)$.  We will show that $p$ is realized in $M$.
Notice that if $\cf(\theta)\geq\mu^+$, the result follows easily.    If $\cf(\theta)\geq\mu^+$, then $N\prec_{\K}M_\alpha$ for some $i<\theta$.  Because $M_i$ is $\mu^{++}$-saturated, $p$ is realized in $M_i$.
 
 So, let us consider the more interesting case that $\cf(\theta)<\mu^+$.   Our goal is to define a sequence of models $\langle M^*_i\mid i<\cf(\theta)\rangle$ inside $M$ so that 
  $M^*_{i+1}$ is universal over $M^*_i$ and so that $M^*:=\Union_{i<\cf(\theta)}M^*_i$ contains $N$.

Suppose for the sake of contradiction that $p$ is omitted in $M$.  Then we can, by increasing the universe of $N$ if necessary, use the Downward L\"{o}wenheim-Skolem axiom to find 
$\langle N_i\in\K_\mu\mid i<\theta\rangle $  an increasing and continuous resolution of $N$ so that $N_i\prec_{\K}M_i$ for each $i<\theta$.

 We define an increasing and continuous sequence $\langle M^*_{i}\mid i<\theta\rangle$  so that for $i<\theta$: 
 \begin{enumerate}
 \item $M^*_{i}\in\K_{\mu^+}$ is  a  limit model.
 \item $N_i\prec_{\K}M^*_{i}$.
 \item $M^*_i\prec_{\K}M_i$.
\item\label{univ over condition} $M^*_{i+1}$ is a universal over $M^*_{i}$.
 \end{enumerate}
  This construction is straightforward since each $M_i$ is $\mu^{++}$-saturated and hence universal over every submodel of cardinality $\mu^+$.
We are assuming $\mu^+$-stability, so limit models of cardinality $\mu^+$ exist.  Therefore  $M_0$ contains a $(\mu^+,\omega)$-limit model containing $N_0$.  Let this be $M^*_0$.  Suppose $M^*_i$ has been defined.  Let $M^{**}$ be a submodel of $M_{i+1}$ of cardinality $\mu^+$ containing  $N_{i+1}\Union M^*_i$.  Because $M^*_{i+1}$ is $\mu^{++}$ saturated, it is $\mu^+$-universal over $M^{**}$, and therefore it contains a model $M^*_{i+1}$ of cardinality $\mu^+$ universal over $M^{**}$.  At limit ordinals $i$, we can take unions since both the sequences $\bar M$ and $\bar N$ are continuous.

 Let $M^*:=\Union_{i<\theta}M^*_i$.    By condition \ref{univ over condition} of the construction,  $M^*$ is a $(\mu^+,\theta)$-limit model.  Since we assume  $\mu^+$-symmetry and $\mu^+$-superstability, we can apply Theorem \ref{uniqueness thm} to conclude that this $(\mu^+,\theta)$-limit model is  $\mu^+$-saturated.  Thus $p$ is realized in $M^*$, and consequently in $M$ as required.

  \end{proof}
  
 A similar proof to Theorem \ref{limit thm} for a result related to Corollary \ref{limit cor} are found in \cite[Theorem 10.22]{Ba}.

     \begin{corollary}\label{limit cor}
 Let $\K$ be an abstract elementary class which satisfies the joint embedding and amalgamation properties.  Suppose $\lambda>\LS(\K)$ is a limit cardinal and $\theta$ is a limit ordinal $<\lambda^+$. If $\K$ is $\mu^+$-superstable  and satisfies $\mu^+$-symmetry for unboundedly many $\mu<\lambda$, then for any increasing and continuous sequence $\langle M_i\mid i<\theta\rangle$  of $\lambda$-saturated models, $\Union_{i<\theta}M_i$ is $\lambda$-saturated.
 \end{corollary}

 \section{Union of Saturated Models}\label{sec:main theorem}
 
 In this section we prove Theorem \ref{main theorem}, by proving a slightly stronger statement.  
 Notice that Theorem \ref{uniqueness thm} and Theorem \ref{saturated theorem} together imply Theorem \ref{main theorem}.

  \begin{theorem}\label{saturated theorem}
 Let $\K$ be an abstract elementary class which satisfies the joint embedding and amalgamation properties.  Suppose $\mu\geq\LS(\K)$ is a  cardinal. If $\K$ is $\mu$- and $\mu^+$-superstable  and satisfies the property that  all limit models of cardinality $\mu^+$ are isomorphic, then for any increasing sequence $\langle M_i\in\K_{\geq\mu^{+}}\mid i<\theta<(\sup\|M_i\|)^+\rangle$  of $\mu^+$-saturated models, 
 $\Union_{i<\theta}M_i$ is $\mu^+$-saturated.
 \end{theorem}
The proof is similar to the proof of Theorem \ref{limit thm}, only here the construction of $\langle M^*_i\mid i<\theta\rangle$ inside $M:=\Union_{i<\theta}M_i$ is a little more nuanced since the cardinality of $M^*_i$ and the cardinality of the saturated models $M_i$ may be the same.  We will be using directed limits, and while we won't arrange that the limit of the directed system of $\langle M^*_i\mid i<\theta\rangle$ lies in $M$, we will get the most critical part, the realization of the type, to lie in $M$.

\begin{proof}
As in the first paragraphs of the proof of Theorem \ref{limit thm}, we may assume without loss of generality that the sequence  $\langle M_i\in\K_{\geq\mu^{+}}\mid i<\theta<(\sup\|M_i\|)^+\rangle$  is continuous and that $\cf(\theta)=\theta<\mu^+$.

Fix $N\in\K_\mu$ with $N \prec_{\K}\Union_{i<\theta}M_i$  and suppose $p\in\gaS(N)$ is omitted in $M:=\Union_{i<\theta}M_i$.  Then, because each $M_{i+1}$ is $\mu^+$-saturated,  we may assume without loss of generality  that $N$ is a $(\mu,\theta)$-limit model witnessed by $\langle N_i\mid i<\theta\rangle$ with $N_i\prec_{\K}M_i$, if necessary by expanding $N$.  Furthermore by $\mu$-superstability we may assume that $p$ does not $\mu$-split over some $\check N$ with $N_0$ a limit model over  $\check N$, by renumbering the sequences $\bar N$ and $\bar M$ if necessary.  

For each $i<\theta$, because $M_i$ is $\mu^+$-saturated, we can find a sequence $\langle \mathring M_i^\alpha\in\K_{\mu}\mid \alpha<\mu^+\rangle$ so that $\mathring M^0_i=N_i$ ,
 $\mathring M^\alpha_i\prec_{\K}M_i$, and $\mathring M^{\alpha+1}_i$ is $\mu$-universal over $\mathring M_i^\alpha$.  Therefore $M_i$ contains a $(\mu,\mu^+)$-limit model, which is 
 isomorphic to  a $(\mu^+,\mu^+)$-limit model by Proposition \ref{mu-plus-limit}.  So, inside each $M_i$ we can find a $(\mu^+,\mu^+)$-limit model witnessed by a sequence that we will denote by $\langle\grave M_i^\alpha\in\K_{\mu^+}\mid \alpha<\mu^+\rangle$,  and  we may arrange the enumeration so that  $N_i\prec_{\K}\grave M^0_i$.

We will build a directed system  of models $\langle M^*_i\mid i<\theta\rangle$ with mappings $\langle f_{i,j}\mid i\leq j<\theta\rangle$ so that the following conditions are satisfied:
\begin{enumerate}
\item $M^*_i\in\K_{\mu^+}$.
\item $M^*_i\preceq_{\K}\Union_{\alpha<\mu^+}\grave M_i^\alpha\preceq_{\K}M_i$. 

\item for $i\leq j<\theta$, $f_{i,j}:M^*_i\rightarrow M^*_j$. 
\item\label{identity condition} for $i\leq j<\theta$, $f_{i,j}\restriction N_i=id_{N_i}$.
\item\label{univ condition direct limit} $M^*_{i+1}$ is universal over $f_{i,i+1}(M^*_i)$.

\end{enumerate}

 Refer to Figure \ref{fig:T*}.
 \begin{figure}[tb]
\begin{tikzpicture}[rounded corners=5mm,scale =2.5,inner sep=.5mm]
\draw (0,0) rectangle (4.5,-.5);
\draw (0,0) rectangle (4.5,-2);
\draw (0,0) rectangle (1,-2);
\draw (0,0) rectangle (2,-2);
\draw (0,0) rectangle (3,-2);
\draw (.85,-.4) node {$N_0$};
\draw (1.75,-.4) node {$\dots N_j$};
\draw (2.8,-.4) node {$N_{j+1}$};
\draw (3.85,-.4) node {$\dots \Union_{i<\theta}N_i=N$};
\draw (.8,-1.9) node {$M_0$};
\draw (1.7,-1.9) node {$\dots M_j$};
\draw (2.7,-1.9) node {$M_{j+1}$};
\draw (3.85,-1.9) node {$\dots \Union_{i<\theta}M_i=M$};
\coordinate (m01-in) at (0,-.9);
\coordinate (m01-out) at (1,-.9);
\draw    (m01-in) to[out=-20,in=200] coordinate[pos=0.7](A1)  (m01-out);
\draw (.6,-.9) node {$M^*_{0}$};
\coordinate (m10-in) at (0,-.5);
\coordinate (m10-out) at (2,-.5);
\draw    (m10-in) to[out=-20,in=240] coordinate[pos=0.7](Ai)  coordinate[pos=.8](ai)(m10-out);
\draw (1.4,-.7) node {$M^*_{j}$};
\coordinate (m1t-in) at (0,-.4);
\coordinate (m1t-out) at (3,-0.5);
\draw    (m1t-in) to[out=-80,in=240] coordinate[pos=0.8](Ai1) (m1t-out);
\draw (2.6,-.7) node {$\grave M^1_{j+1}$};
\draw [->, shorten >=3pt] (A1) to [bend right=65] node[pos=0.7,below] {$f_{0,j}$}(Ai);
\draw [->, shorten >=3pt] (ai) to [bend right=25] node[pos=0.7,above] {$f_{j,j+1}$}(Ai1);
\draw    (m1t-in) to (.2, -1.5) to (2.2,-1.8) to coordinate[pos=0.8](Ai1) (m1t-out);
\draw (1.5,-1.55) node {$\grave M^2_{j+1}=M^*_{j+1}$};
\end{tikzpicture}
\caption{The directed system in the proof of Theorem \ref{saturated theorem}.} \label{fig:T*}
\end{figure}
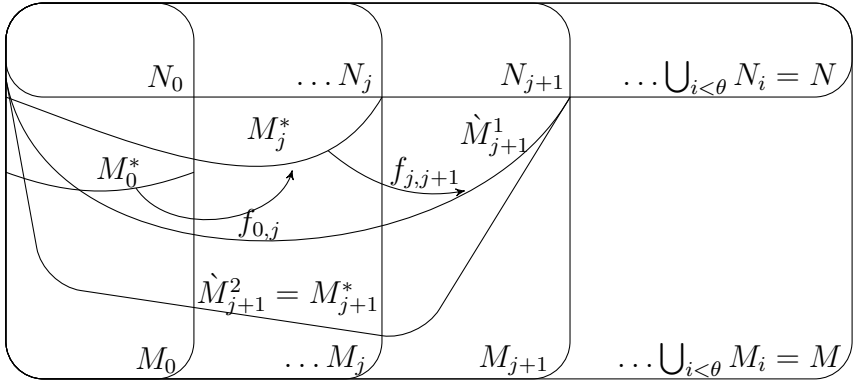

The construction is possible.  Take $M^*_0$ to be $\grave M_0^1$ and $f_{0,0}=\id$.  At limit stages take $M^{**}_i$ and $\langle f^{**}_{k,i}\mid k<i\rangle$ to be a direct limit as in Fact \ref{direct limit lemma}.  We do not immediately get that $M^{**}_i\preceq_{\K}M_i$; we just know we can choose $M^{**}_i$ to contain $N_i$ by the continuity of $\bar N$ and condition \ref{identity condition} of the construction.  We also know by condition \ref{univ condition direct limit} that $M^{**}_i$ is a $(\mu^+,i)$-limit model  witnessed by $\langle f_{k,i}(M^*_k)\mid k<i\rangle$.
By our assumption of the uniqueness of limit models of cardinality $\mu^+$,
$M^{**}_i$ is a $(\mu^+,\mu^+)$-limit model.  Since $N_i$ has cardinality $\mu$, being able to write $M^{**}_i$ as a $(\mu^+,\mu^+)$-limit model tells us that $M^{**}_i$ is $\mu^+$-universal over $N_i$.  Recall that  $\Union_{\alpha<\mu^+}\grave M^\alpha_i$ is also a $(\mu^+,\mu^+)$-limit model containing $N_i$.  Therefore, we can find an isomorphism $g$ from $M^{**}_i$ to $\Union_{\alpha<\mu^+}\grave M^\alpha_i$ fixing $N_i$.  Now take $M^*_i:=g(M^{**}_i)=\Union_{\alpha<\mu^+}\grave M^\alpha_i$, $f_{k,i}:=g\circ f^{**}_{k,i}$ for $k<i$, and $f_{i,i}=\id$.  

For the successor stage of the construction, assume that $M^*_j$ and $\langle f_{k,j}\mid k\leq j\rangle$ have been defined.  Since $M^*_j$ is a model of cardinality $\mu^+$ containing $N_j$ and because $\grave M^{1}_{j+1}$ is $\mu^+$-universal over $N_{j+1}$ we can find a embedding $g:M^*_j\rightarrow \grave M^1_{j+1}$ with $g\restriction N_j=\id_{N_j}$.  Take $M^*_{j+1}:=\grave M^2_{j+1}$,  set $f_{k,j+1}:=g\circ f_{k,j}$ for all $k\leq j$, and define $f_{j+1,j+1}:=\id$.  This completes the construction.

Take $M^*$ with mappings $\langle f_{i,\theta}\mid i<\theta\rangle$  to be the direct limit of the system as in Fact \ref{direct limit lemma}.  While $M^*$ may not be inside $M$, we can arrange that $f_{i,\theta}\restriction N_i=\id_{N_i}$ and that $N\prec_{\K}M^*$.  Notice that by condition \ref{univ condition direct limit} of the construction, $M^*$ is a $(\mu^+,\theta)$-limit model.   From our assumption of the uniqueness of $\mu^+$-limit models and Proposition \ref{mu-plus-limit}, we can conclude that $M^*$ is saturated.  

For each $i<\theta$, let $f^*_{i,\theta}\in\Aut(\C)$ extend $f_{i,\theta}$ so that $f^*_{i,\theta}(N)\preceq_{\K}M^*$.  This is possible since we know that $M^*$ is $\mu^+$-universal over $f_{i,\theta}(M^*_i)$ by condition \ref{univ condition direct limit} of the construction.  
Let $N^*\prec_{\K}M^*$ be a model of cardinality $\mu$ extending $N$ and $\Union_{i<\theta} f^*_{i,\theta}(N)$.  By the extension property for non-$\mu$-splitting, we can find $p^*\in\gaS(N^*)$ extending $p$ so that 
\begin{equation}\label{p*}
p^*\text{ does not }\mu\text{-split over }\check N.
\end{equation}
  Since $M^*$ is a saturated model of cardinality $\mu^+$ containing the domain of $p^*$,
we can find $b^*\in M^*$ realizing $p^*$.  By the definition of a direct limit, there exists $i<\theta$ and $b\in M^*_i$ so that $f_{i,\theta}(b)=b^*$.  

Because $f_{i,\theta}\restriction N_i=id_{N_i}$, we know that $b\models p\restriction N_i$.  Suppose for sake of contradiction that there is some $j>i$ so that $\tp(b/N_j)\neq p\restriction N_j$.  Then, by the uniqueness of non-splitting extensions, it must be the case that $\tp(b/N_j)$ $\mu$-splits over $\check N$.  By invariance, 
\begin{equation}\label{non-split equation}
\tp(f_{i,\theta}(b)/f^*_{i,\theta}(N_j)) \;\mu\text{-splits over }\check N.
\end{equation}  By monotonicity of non-splitting,  the definition of $b$, and choice of $N^*$ containing $f^*_{i,\theta}(N)$, $(\ref{non-split equation})$ implies $\tp(b^*/N^*)$ $\mu$-splits over $\check N$.  This contradicts $(\ref{p*})$.

Since $b\models p\restriction N_j$ for all $j<\theta$ and $p\restriction N_j$ does not $\mu$-split over $\check N$, $\mu$-superstability implies that $\tp(b/N)$ does not $\mu$-split over $\check N$.  By uniqueness of non-$\mu$-splitting extensions $\tp(b/N)=p$.  Since $b\in M_i$, we are done.

\end{proof} 
 
%

\section{Concluding Remarks}
The characterization of $\mu$-symmetry by reduced towers in \cite{Va3-SS} spawned many results during the summer of 2015, including the work here.  While these new results deal with some of the same concepts (towers, superstability, limit models, union of saturated models), the contexts and methods differ.  The focus here is in local properties of the classes $\K_\mu$ and $\K_{\mu^+}$ without assuming categoricity, tameness, or sufficiently large cardinals.  In this section, we summarize how some of the other results relate to Theorem \ref{main theorem} and Theorem \ref{symmetry transfer}.

Most closely related to Theorem \ref{symmetry transfer} is \cite{VV} where the authors develop a more nuanced technology of towers.  
The structure of the proof of Theorem \ref{symmetry transfer} involves taking a tower $\T\in\K^*_{\mu,\alpha}$ and building from it a tower in $\K^*_{\mu^+,\alpha}$.  VanDieren and Vasey show that it is possible to carry out this kind of construction to produce a tower in $\K^*_{\lambda,\alpha}$ for $\lambda>\mu^+$ \cite{VV} if one assumes $\kappa$-superstability for an interval of cardinals.  The consequent improvements of Theorem \ref{symmetry transfer} and its corollaries to more global properties of the class are explored in \cite{VV}.  
Another paper using this technology of towers is \cite{Bo-Van} in which the authors, Boney and VanDieren, study the implications of Theorem \ref{main theorem} and Theorem \ref{symmetry transfer}  in classes that are $\mu$-stable but not $\mu$-superstable.

In just a few months after the introduction of $\mu$-symmetry and its equivalent formulation and the announcement of Theorem \ref{main theorem}, several advances have been made.  Theorem \ref{main theorem} has broken down a door in the development of a classification theory for abstract elementary classes assuming additional properties on the class like tameness or additional structural properties like categoricity.  
VanDieren and Vasey examine Theorem \ref{main theorem} in \emph{tame} abstract elementary classes and use it to show the existence of a unique type-full good $\mu^+$-frame in a $\mu$-superstable, $\mu$-tame AEC \cite{VV}.  This analysis is then used by VanDieren and Vasey to improve structural results for  AECs categorical in a sufficiently large cardinality.  For example,  
they show
 that for $\K$  an AEC with no maximal models and $\mu$ is a cardinal $\geq\LS(\K)$,  if $\K$ is categorical in a $\lambda\geq h(\mu^+)$, then the model of size $\lambda$ is $\mu^+$-saturated \cite{VV2}.
The union of saturated models is saturated is employed by Vasey  to prove the equivalence of the existence of prime models and categoricity in a tail of cardinals  in categorical, tame, and short AECs \cite{V-prime}.  Furthermore, Vasey in \cite{V-downward} uses Theorem \ref{main theorem} in a crucial way
to lower the bound, from the second Hanf number down to the first, on the categoricity cardinal in Shelah's seminal Downward Categoricity Theorem for AECs  \cite{Sh394}.
Additionally, VanDieren has examined the proofs of Theorem \ref{symmetry transfer} and Theorem \ref{main theorem} in categorical AECs in which the amalgamation property is not assumed \cite{Va-char}, providing additional insight into Shelah and Villaveces' original exploration of limit models \cite{ShVi}.


 \section*{Acknowledgement}
 The author is grateful to Sebastien Vasey for email correspondence about \cite{Va3-SS} during which he asked her  about the union of saturated models.
 She is also thankful to Rami Grossberg, William Boney, Sebastien Vasey, and the referees  for suggestions and comments that improved the clarity of this paper.   
 
\bibliographystyle{elsarticle-harv}
\bibliography{union-of-sat-transfer-12.18.bib}

\end{document}